\documentclass[[amscd,amssymb,verbatim,12pt]{amsart}
\usepackage{amssymb, latexsym, amsmath, amscd, array, graphicx}
\usepackage{hyperref}
\usepackage[all]{xy}

\swapnumbers \numberwithin{equation}{section}
\usepackage{tikz-cd}

\theoremstyle{plain}

\newtheorem{thm}{Theorem}[section]
\newtheorem{theorem}[thm]{Theorem}

\newtheorem{lemma}[thm]{Lemma}

\newtheorem{prop}[thm]{Proposition}
\newtheorem{cor}[thm]{Corollary}

\theoremstyle{definition}
\newtheorem{defn}[thm]{Definition}

\newtheorem{remark}[thm]{Remark}

\newtheorem{question}[thm]{Question}

\def\im{\protect\operatorname{im}}

\DeclareMathOperator{\cd}{{\rm cd}}

\DeclareMathOperator{\TC}{{\rm TC}}

\DeclareMathOperator{\supp}{{\rm supp}}
\DeclareMathOperator{\dTC}{{\rm dTC}}
\DeclareMathOperator{\ATC}{{\rm ATC}}

\def\Int{\protect\operatorname{Int}}

\def\cat{\protect\operatorname{cat}}
\def\acat{\protect\operatorname{acat}}

\def\Z{{\mathbb Z}}

\def\1{\hbox{\rm\rlap {1}\hskip.03in{\rom I}}}
\def\Bbbone{{\rm1\mathchoice{\kern-0.25em}{\kern-0.25em}
{\kern-0.2em}{\kern-0.2em}I}}

\long\def\forget#1\forgotten{} %

\newcommand\ver[1]{\marginpar{\tiny Changed in Ver \VER}}

\date{\today}

\begin{document}

\title[Distributional TC]{Distributional Topological Complexity of groups}

\author[A.~Dranishnikov]{Alexander  Dranishnikov$^{1}$} 


\thanks{$^{1}$Supported by Simons Foundation}

\address{Alexander N. Dranishnikov, Department of Mathematics, University
of Florida, 358 Little Hall, Gainesville, FL 32611-8105, USA}
\email{dranish@ufl.edu}



\begin{abstract}
We study numerical invariants $d\TC(\Gamma)$ and $d\cat(\Gamma)$ of groups recently introduced in~\cite{DJ} and independently in~\cite{KW}. 
We compute $d\TC$ for finite cyclic groups $\mathbb Z_p$ with prime $p$ as well as
for nonorientable surfaces of genus $g>3$ (for orientable surfaces it was computed in~\cite{DJ}). We prove the formula
$$d\TC(G\ast H)=\max\{d\TC (G),d\TC (H), \cd(G\times H)\}$$ for torsion free groups.
\end{abstract}



\maketitle

\section{Introduction}
The topological complexity of a configuration space $X$ is a numerical invariant $\TC(X)$ which appeared in topological robotics~\cite{Fa2}.
It is closely related to an old numerical invariant called the Lusternik-Schnirelmann categpory~\cite{LS} of a space $X$, $\cat (X)$. Since $\TC(X)$ and $\cat(X)$ are homotopy
invariant they bring numerical invariant of discrete groups defined as $\TC(\Gamma)=\TC(B\Gamma)$ and $\cat(\Gamma)=\cat(B\Gamma)$ where $B\Gamma=K(\Gamma,1)$ is a classifying space for 
the group $\Gamma$. 

In 50s  Eilenberg and Ganea proved~\cite{EG} that the LS-category of a group $\Gamma$ coincides with its cohomological dimension $\cat(\Gamma)=\cd(\Gamma)$.
An algebraic description of $\TC(\Gamma)$ is still missing. What is known that $\cd(\Gamma) \le\TC(\Gamma)\le\cd(\Gamma\times\Gamma)$. This implies in particular that
$\TC(\Gamma)=\infty$ for groups with torsions. It turns out that it is quite difficult to compute the topological complexity of groups $\TC(\Gamma)$~\cite{FM},\cite{GLO}.
There are very few exact computations are known: $\TC(A)=\cd(A)$ for a free abelian group $A$, $\TC(F)=2$ for free non-abelian groups, and 
$\TC(M_g)=4$, $g>1$ for the fundamental group of orientable surfaces (all done in~\cite{Fa1}), for the fundamental group of non-orientable surfaces of genus $g>1$, $\TC(N_g)=4$~\cite{CV} (and~\cite{Dr1} for $g>3$), $\TC(\Gamma)=2\cd(\Gamma)$ for torsion free hyperbolic groups~\cite{Dr3} and for certain toral relatively hyperbolic groups~\cite{Li}.

Recently we in~\cite{DJ} and independently  Ben Knudsen and Shmuel Weinberger  in~\cite{KW}  have defined a new  probabilistic versions of the Lusternik-Schnirelmann category and the topological complexity  of a topological space $X$. Though we were motivated by topological robotics, we introduced new numerical homotopy invariants of spaces and, hence, 
new numerical invariants of discrete groups. In this paper we attempted a further study of these numerical invariants for groups. We call our invariants the distributional topological complexity, denoted as $d\TC$, and the distributional LS-category, denoted as $d\cat$. Knudsen and Weinberger called them analog topological complexity and the analog LS-category and denoted them by
$\ATC$ and $\acat$ respectively. Formally there is a difference between our definitions which is in the choice of topology on the space $\mathcal B_n(Z)$ of probability measures $\mu$ on a topological space $Z$ with the cardinality of supports bounded by $n$,  $|\supp(\mu)|\le n$. The elements of $\mathcal B_n(Z)$ can be viewed as finite linear combinations $\mu=\sum_{z\in Z}\lambda_zz$ of points in $Z$
with $\sum\lambda_z=1$ and $\lambda_z\ge 0$.

There are several ways to introduce topology on $\mathcal B_n(Z)$. The finest topology is the quotient topology which comes from the symmetric join product
 by means of a quotient map $q:Symm^{\ast n}(X)\to \mathcal B_n(Z)$~\cite{KK}. By the definition
$Symm^{\ast n}(Z):=\ast^nZ/S_n$ is the orbit space of action of the symmetric group $S_n$ on the iterated join product $\ast^nZ=Z\ast\cdots\ast Z$. 
 Elements of $Symm^{\ast n}(Z)$ can be seen as formal sums $t_1x_1+\cdots t_nx_n$ with no order 
on the summands and the convention $0x=0y$ for each summand. Then the quotient map $q:Symm^{\ast n}(X)\to \mathcal B_n(Z)$  is defined by the condition $tx+t'x=(t+t')x$. We note that $q$ is a map with compact contractible fibers.

Knudesn and Weinberger used the quotient topology which is non-metrizable 
but it is the standard choice in the definition of $\mathcal B_n(Z)$ (see~\cite{KK}). We decided that a metric topology is more appropriate for robotics and chose
the Levy-Prokhorov metric $d_{LP}$ from several known metrics on measures. Clearly the identity map $(\mathcal B_n(Z),\tau_q)\to(\mathcal B_n(Z),d_{LP})$ is continuous~\cite{J}.
Moreover, it is possible to show that this map is a homotopy equivalence when $Z$ is a locally finite CW-complex. We are not giving a proof of it, since in all  proofs of our results in this paper any choice of topology is good. In this paper we will be using the notations from~\cite{DJ} for the invariants. 

We note when  $Z$ is a discrete space, $\mathcal B_n(Z)$ is the $n$-skeleton of  the simplex $\Delta(Z)$ spanned by $Z$
with the CW-complex topology in the first case and with the metric topology taken from the Hilbert space $\Delta(Z)\subset \ell_2(Z)$ in the second case. Since these topologies agree on finite subcomplexes, the identity map is a weak homotopy equivalence and, since we are dealing with ANR-spaces, it is a homotopy equivalence.

Here are the definitions.
{\em The distributional topological complexity}, $d\TC(X)$, of a space $X$ is the minimal number $n$ such that there is a continuous map $$s:X\times X\to \mathcal B_{n+1}(P(X))$$ satisfying $s(x,y)\in \mathcal B_{n+1}(P(x,y))$ for all $(x,y) \in X \times X$ where $P(x,y)\subset P(X)$ is the set of all paths  in $X$ from $x$ to $y$.

{\em The distributional LS-category}, $d\cat(X)$, of a space $X$ is the minimal number $n$ such that there is a continuous map $$s:X\to \mathcal B_{n+1}(P(X))$$ satisfying $s(x)\in \mathcal B_{n+1}(P(x,x_0))$ for all $x \in X$. 

Like in the case of classical invariants $\TC$ and $\cat$ there are inequalities
$$
d\cat(\Gamma)\le d\TC(\Gamma)\le d\cat(\Gamma\times\Gamma).
$$
Knudsen and Weinberger pointed out on more similarity by proving the Eilenberg-Ganea equality $d\cat\Gamma=\cd(\Gamma)$ for torsion free groups.
The classical Eilenberg-Ganea theorem~\cite{EG} gives the equality $\cat\Gamma=\cd(\Gamma)$ for all groups $\Gamma$.
The striking difference was found in both~\cite{DJ},~\cite{KW} $$d\cat(\mathbb Z_2)=d\TC(\mathbb Z_2)=1.$$
We recall that $\cd(\Gamma)=\infty$ for every finite group.
In contrast, it was proven  in~\cite{KW} that $d\TC(G)\le |\Gamma|-1$ for any finite group $\Gamma$. 

Our main result of this paper is the following.

\

{\bf Theorem A.} (Theorem~\ref{main 1}) {\em For any prime} $p$,
$$d\cat(\mathbb Z_p)=d\TC(\mathbb Z_p)=p-1.$$

Our next result is the free product formula.

\

{\bf Theorem B.} (Theorem~\ref{free}) {\em For torsion free groups,}
$$d\TC(G\ast H)=\max\{d\TC (G),d\TC (H), \cd(G\times H)\}.$$

In~\cite{DJ} we proved that for the fundamental group of orientable surface groups $d\TC(\pi_1(M_g))=d\TC(M_g)=4$ when $g>1$.
In this paper we proved the following theorem for non-orientable surface groups.

\

{\bf Theorem C.} (Theorem~\ref{nonor}) $d\TC(N_g)=4$ for $g>3$.

\

Some of the auxiliary results in the paper could be of interest on their own. One of them is an extension of Singhof's theorem on dimension of categorical sets
to arbitrary simplicial complexes (Theorem~\ref{f-categorical}). Another is the statement about connectivity of $\mathcal B_n(X)$ for general CW complexes $X$ (Theorem~\ref{connectivity} and Theorem~\ref{connectivity1}). We were not able to find corresponding results in the literature.
In a relatively recent paper~\cite{KK} the statement about connectivity of $\mathcal B_n(X)$ improves the classic result of Nakaoka~\cite{Na} but it covers only the 
case of simply connected $X$.

\section{Preliminaries}

\subsection{The LS-category.} 
The {\emph{Lusternik-Schnirelmann category}, $\cat (X)$, of  $X$ is the least number $n$ such that there is a covering  
$\{U_i\}$ of $X$ by $n+1$ open sets each of which is contractible in $X$.

Let $P_0(X)$ be the space of paths in $X$ ending at the  base point $x_{0} \in X$.
 Let $p_0=p_0^X: P_{0}(X) \to X$ be the evaluation fibration $p_0(\phi)= \phi(0)$. Then we define the $n^{th}$ Ganea space, denoted $G_{n}(X)$, to be the fiberwise join of  $(n+1)$-copies of $P_{0}(X)$ along $p_0$, i.e.,
$$
G_{n}(X) = \left\{\sum_{i=1}^{n+1} \lambda_{i} \phi_{i} \hspace{1mm} \middle| \hspace{1mm} \phi_{i} \in P_{0}(X), \sum_{i=1}^{n+1}\lambda_{i} = 1, \lambda_{i} \geq 0, \phi_{i}(0) = \phi_{j}(0)\right\}.
$$
We denote by  $p_{n}^{X} : G_{n}(X) \to X$ the $n^{th}$ Ganea fibration, where  $p_{n}^{X} (\sum\lambda_{i} \phi_{i}) = \phi_{i}(0)$, for any $i$ such that $\lambda_{i} > 0$. Thus, $p^X_0=p_0$. The following theorem gives the Ganea-Schwarz characterization of the LS-category~\cite{Sch}, \cite{CLOT}.

\begin{theorem}\label{original}
    For any $X$, $\textup{\text{cat}}(X) \leq n$ if and only if the fibration $p_{n}^{X}: G_{n}(X) \to X$ admits a section.
\end{theorem}
\subsection{The Topological Complexity (TC)}
Let $P(X)$ be the space of all paths in $X$ and  let $\bar p=\bar p^X: P(X) \to X\times X$ denote the end-points fibration $\bar p(\phi)=(\phi(0),\phi(1))$.
The \emph{topological complexity}, $\TC(X)$, of  $X$ is the least number $n$  such that there is a covering $\{U_i\}$ of $X\times X$ by $n+1$ open sets where each of which  
admits a motion planning algorithm. We recall that a \emph{motion planning algorithm} over an open subset $U\subset X\times X$ is a section $U\to P(X)$ of $\bar p^X$.
The following is straightforward.
\begin{prop}
A set $U\subset X\times X$ admits a motion planning algorithm if and only if it admits a deformation in $X\times X$ to the diagonal $\Delta X$, i.e. there is a homotopy $H:U\times I\to X\times X$
such that $H((x,y),0)=(x,y)$ and $H((x,y),1)\in\Delta X$ for all $(x,y)\in U$.
\end{prop}

We define the $n^{th}$ Schwarz-Ganea space, denoted $\Delta_{n}(X)$, to be the fiberwise join of  $(n+1)$-copies of $P(X)$ along $\overline{p}$, i.e.,
$$
\Delta_{n}(X) = \left\{\sum_{i=1}^{n+1} \lambda_{i} \phi_{i} \hspace{1mm} \middle| \hspace{1mm} \phi_{i} \in P(X), \sum_{i=1}^{n+1}\lambda_{i} = 1, \lambda_{i} \geq 0, \bar p(\phi_{i})=\bar  p(\phi_{j})\right\}.
$$
 We define the $n^{th}$ Schwarz-Ganea fibration, $\bar p_{n}^{X} : \Delta_{n}(X) \to X\times X$, as  $$\bar p_{n}^{X} \left(\sum_{i=1}^{n+1} \lambda_{i} \phi_{i}\right) = (\phi_{i}(0),\phi_{i}(1))$$ for any $i$ with $\lambda_{i} > 0$. Then the following theorem gives the Ganea-Schwarz characterization of the topological complexity \cite{Sch}.

\begin{theorem}\label{original2}
    For any $X$, $\TC(X) \leq n$ if and only if the fibration $\bar p_{n}^{X}: \Delta_{n}(X) \to X\times X$ admits a section.
\end{theorem}
The invariants $\cat $ and $\TC$ are defined for maps $f:X\to Y$ as follows. The LS-category $\cat (f)$ is the least number $n$ such that there is a covering  
of $X$ by $n+1$ open sets $\{U_i\}_{i=0}^n$ such that the restrictions $f|_{U_i}$ are null-homotopic for all $i$~\cite{Fo}.
The topological complexity $\TC(f)$ is the least number $n$ such that there is a covering  
of $X\times X$ by $n+1$ open sets $\{U_i\}_{i=0}^n$ such that for each $i$ the restriction $f|_{U_i}$ is homotopic to a map with the image in the diagonal $\Delta Y$~\cite{Sco}.
Note that $\cat(1_X)=\cat(X)$ and $\TC(1_X)=\TC(X)$.

Theorem~\ref{original} and Theorem~\ref{original2} can be extend to the following.
\begin{prop}\label{liftCL}
Let $f:X\to Y$ be a map. Then

(a)~\cite{Dr1}   $\cat(f)\le n$ if and only if $f$ admits a lift with respect to $p_n^Y$;

(b)~\cite{Sco}  $\TC(f)\le n$ if and only if $f\times f$  admits a lift with respect to $\bar p_n^Y$.
\end{prop}

We will be using the following formula from~\cite{DS}.
\begin{thm}\label{wedgeCL}
The equality
$$\TC(X\vee Y)=\max\{\TC (X),\TC (Y),\cat(X\times Y)\}$$
holds for CW complexes $X$ and $Y$
whenever $$\max\{\dim X,\dim Y\}<\max\{\TC (X),\TC (Y),\cat(X\times Y)\}.$$
\end{thm}

\subsection{Distributive versions of the LS-category and TC}
For any map $p:E\to B$ and $n\in\mathbb N$, we define the map $\mathcal B_n(p):E_n\to B$ as follows. Let $$E_n=\{\mu\in\mathcal B_n(E)\mid \supp\mu\subset p^{-1}(x), x\in B\}$$ 
denote the result of the fiberwise application of the functor $\mathcal B_n$ to $E$.
  We define
$\mathcal B_n(p)(\mu)=x$ for $\mu\in \mathcal B_n(p^{-1}(x))$. We proved in~\cite{DJ} that $\mathcal B_n(p)$ is a Hurewicz fibnration if so is $p$.

The following characterizations of $d\cat$ and $d\TC$ are taken from~\cite{DJ}.
\begin{thm}\label{chardcat}
$d\cat(X) \le n$ if and only if the fibration $$\mathcal B_{n+1}(p_{0}):P_0(X)_{n+1}\to X$$ admits a section.
\end{thm}

\begin{thm}\label{chardTC}
$d\TC (X)\le n$ if and only if the fibration $$\mathcal B_{n+1}(\bar p):P(X)_{n+1}\to X\times X$$ admits a section.
\end{thm}

We define the invariants $d\cat$ and $d\TC$ for  maps as follows. {\em The distributive topological complexity}, $d\TC(f)$ , of a map $f:X\to Y$ is the minimal number $n$ such that there is a continuous map
$s:X\times X\to\mathcal B_{n+1}(P(Y))$ satisfying $s(x,x')\in\mathcal B_{n+1}(P(f(x),f(x'))$ for all $(x,x')\in X\times X$.
Here $P(y,y')=\{f\in P(Y)\mid f(0)=y,\ f(1)=y'\}$.

{\em The distributional LS-category}, $d\cat(f)$,  is the minimal number $n$ such that there is a continuous map $$s:X\to \mathcal B_{n+1}(P(Y))$$ satisfying $s(x)\in \mathcal B_{n+1}(P(f(x),y_0))$ for all $x \in X$. We note that $$d\cat(1_X)=d\cat(X)\ \ \text{and}\ \ d\TC(1_X)=d\TC(X).$$ Clearly, for $f:X\to Y$,
$$d\cat(f)\le\min\{d\cat X, d\cat Y\}\  \text{and} \ d\TC(f)\le\min\{d\TC(X),d\TC(Y)\}.$$
The proof of Theorem~\ref{chardcat} and Theorem~\ref{chardTC} in~\cite{DJ} can be extended to the following.
\begin{prop}\label{liftD}
Let $f:X\to Y$ be a map. Then

(a)   $d\cat(f)\le n$ if and only if $f$ admits a lift with respect to $\mathcal B_{n+1}(p_0^Y)$;

(b) $d\TC(f)\le n$ if and only if $f\times f$  admits a lift with respect to $\mathcal B_{n+1}(\bar p^Y)$.
\end{prop}

We note that in both fibrations $p_0$ and $\bar p$ the fiber is homotopy equivalent to the loop space $\Omega X$.

\subsection{Lower bounds}
Let $R$ be a commutative ring and $f:X\to Y$ be a map.
The {\em $R$-cup-length}t $c\ell_R(f^*)$ of $f$ is by definition the
maximal $k$ such that $$f^*(\alpha_1\smile\dots\smile \alpha_k)\ne 0$$ for $\alpha_i\in H^{n_i}(Y;R)$, $n_i>0$,  $i=1,\dots, k$.
The maximal such $k$ when $\alpha_i\in  H^{n_i}(Y;R_i)$ where $R_i$ is an arbitrary $\pi_1(Y)$-module is called the {\em cup-length} of $f$ and is denoted by $c\ell(f^*)$. We note that in this case  the product of $\alpha_i$  lives in $H^*(Y;R_1\otimes\cdots\otimes R_k)$.

The following proposition has a standard proof (see Exercise 1.16(3) in\cite{CLOT}).

\begin{prop}\label{cuplengthlowbound}
For any map $f:X\to Y$, $$\cat(f)\ge c\ell(f^*).$$
\end{prop}
For $f=1_X$ this is the standard cup-length lower bound for $\cat(X)$.
We recall that $c\ell B\Gamma=\cd(\Gamma)$. Namely, if $\cd(\Gamma)=n$
then $\beta_\Gamma^n\ne 0$ where $\beta_\Gamma\in H^1(\Gamma,I(\Gamma))$ is the Berstein-Schwarz class of $\Gamma$~\cite{DR}.

\

Let $\delta_m:X\to SP^n(X)$ denote the diagonal embedding into the symmetric product of $m$ copies of $X$, $SP^m(X)=X^m/S_m$,  where $S_m$ is the $m$-th symmetric group.  We denote by $[x_1,\dots,x_k]$ the $S_m$-orbit of $(x_1,\dots,x_k)$. 

In~\cite{DJ} we proved that the rational cup-length of a CW-complex $X$ is a lower bound for $d\cat X$. The proof is based on the following
\begin{lemma}\label{SP}
If $d\cat X<n$ then there is an open cover $\{U_i\}_{i=1}^n$ of $X$ such that
each map $\delta_i:U_i\to SP^i(X)$ is null-homotopic.
\end{lemma}
Since the natural map $\delta^k_{nk}:SP^k(X)\to SP^{nk}(X)$ defined as $$\delta^k_{nk}([x_1,\dots x_k])=[x_1,\dots, x_k,x_1,\dots, x_k,\cdots,x_1,\dots, x_k]$$
takes diagonal to diagonal,
we obtain the following.

\begin{thm}\label{diagonal}\cite{DJ}
If $d\cat X<n$, then $d\cat(X)\ge \cat(\delta_{n!})$.
\end{thm}

Then Proposition~\ref{cuplengthlowbound}  implies the following.
\begin{cor}\label{catdiag}
If $d\cat X<n$, then $d\cat(X)\ge c\ell(\delta_{n!}^*)$.
\end{cor}
Let $x_0\in X$ be a base point $x_0$. Then there is the inclusion $\xi_k:X\to SP^k(X)$ defined as $\xi_k(x)=[x,x_0,\dots,x_0]:=S_m(x,x_0,\dots x_0)$.
Moreover, there are inclusions $\xi_{k+i}^k:SP^k(X)\to SP^{k+i}(X)$ defined as $\xi_{k+i}^k([x_1,\dots, x_k])=[x_1,\dots, x_k, x_0,\dots, x_0]$. Then $\xi_k=\xi_k^1$
Note that $\xi_{k+i+j}^{k+i}\circ\xi^k_{k+i}=\xi^k_{k+i+j}$. Thus, the direct limit $SP^\infty(X)=\lim_{\rightarrow}SP^k(X)$ is well-defined. The space $SP^\infty(X)$ is known as free abelian topological monoid generated by $X$. The Dold-Thom theorem states that for a CW-complex $X$ the induced map for homotopy groups $$(\xi_\infty)_*:\pi_i(X)\to \pi_i(SP^\infty(X))=H_i(X)$$ is the Hurewicz homomorphism~\cite{DT}. This holds true for $\xi_k$ and $i< k$.
The following is well-known.
\begin{prop}\label{times m}
For the homology induced homomorphisms $$(\delta_k)_*=k(\xi_k)_*.$$
\end{prop}

\begin{prop}\label{BerSchw}
Let $L^n_p$ be the $n$-skeleton of the infinite lens space $L^\infty_p=S^\infty/\mathbb Z_p$ with respect to the standard CW-complex structure. Then  $\cat(\xi_m)\ge n$
for the map $\xi_m:L^n_p\to SP^m(L^n_p)$ for all $m$.
\end{prop}
\begin{proof}
We show that $\cat(\xi_m)\ge c\ell(\xi_m^*)$.
Since $\pi_1(L^n_p)$ is abelian, by the Dold-Thom theorem the map $\xi_m$
induces an isomorphism of the fundamental groups. Therefore the inclusion
$i:L^n_p\to L^\infty_p$ homotopically factors through $\xi_m$, i.e., there is a map
$g:SP^m(L^n_p)\to L^\infty_p$ such that $i$ is homotopic to $g\xi_m$.
Since $i^*(\beta^n_{\mathbb Z_p})\ne 0$, where $\beta_{\mathbb Z_p}$ is the Berstein-Schwarz class, we obtain $\xi_m^*((g^*\beta_{\mathbb Z_p})^n)\ne 0$.
\end{proof}

The standard lower bound for $\TC(X)$ is given in terms of zero divisor cup-length $zc\ell(H^*(X\times X;R))$~\cite{Fa1} which can be extended to a lower bound for maps $f:X\to Y$
as in~\cite{Sco} $$\TC(f)\ge zc\ell_R((f\times f)^*(H^*(Y\times Y;R))).$$

Similarly one can prove the following theorem and corollary(see~\cite{DJ})
\begin{thm}\label{diagonal2}
If $d\TC X<n$, then $d\TC(X)\ge \TC(\delta_{n!})$ for the diagonal inclusion $\delta_{n!}:X\times X\to SP^{n!}(X\times X)$.
\end{thm}

\begin{cor}
If $d\TC X<n$, then $d\TC(X)\ge zc\ell_R(\im\delta_{n!}^*)$ where $\delta_{n!}:Y\times Y\to SP^{n!}(Y\times Y)$ is the diagonal embedding.
\end{cor}
We don't supply all details here, since this lowe bound is not used in this paper.

\

The following proposition is analogous to one from the classical LS-category.
\begin{prop}\label{covering}~\cite{DJ}
Let $p:X\to Y$ be a covering map, then $d\cat Y\le d\cat X$.
\end{prop}

\subsection{Categorical sets} Let $f:X\to Y$ be a map, a set $A\subset X$ is called {\em $f$-categorical} if $f|_A:A\to Y$ is null-homotopic. Thus, $\cat(f)\le n$ if and only if $X$ can be covered
by $n+1$ $f$-categorical open sets. If $X$ is a CW complex, then the requirement on categorical sets to be open can be dropped~\cite{Sr}.

We call a union $X_1\cup\dots\cup X_k=X$ {\em a partition of} $X$ if $\Int X_i\cap \Int X_j=\emptyset$ for $i\ne j$ and the closure of $\Int X_i$ equals $X_i$ for all $i$.

\begin{thm}\label{f-categorical}
Let $f:X\to Y$ be a map of a finite simplicial complex $X$ of $\dim X=n$ which admits a partition 
$X=V^0_1\cup\dots\cup V^0_k$ into $k$ $f$-categorical subcomplexes. Then $X$  admits a partition $X=P_1\cup \dots\cup P_k$ into $f$-categorical subcomplexes with respect to some subdivision of $X$ such that each complex $P_i$ admits a deformation retraction onto an $(n-k+1)$-dimensional subcomplex.
\end{thm}
\begin{proof}
In the case when $X$ is a manifold this theorem was proved by Singhof~\cite{Si}.

We define a transformation of the partition $X=V^0_1\cup\dots\cup V^0_k$ into a 
partition $X=V^1_1\cup\dots\cup V^1_k$ of $f$-categorical sets which are simplicial sucomplexes of the second barycentric subdivision $\beta^2X$ such that $V^1_1$ admits a deformation retraction onto an $(n-k+1)$-dimensional subcomplex. Then we take the cyclic permutation $i\to i-1$ on the index set $1,2,\dots,k$ and apply our transformation again. After applying this transformation $k$ times we obtain our 
partition $X=P_1\cup \dots\cup P_k$. Note that $P_i$ are subcomplexes of $\beta^{2k}X$,
the $2k$-iterated barycentric subdivision of $X$.

Let $b_\sigma$ denote the barycenter of a simplex $\sigma\subset X$.
For $i=2,\dots,k$ we denote by $B_i$ the union of all stars $St(b_\sigma, \beta^2V^0_1)$ of the barycenters of simplices $\sigma\subset V^0_1$ of $\dim\sigma=n-k+i$. Note that
$B_i$ is the disjoint union of these stars. We define 
$$V^1_1=V^0_1\setminus(\bigcup_{i=2}^k\Int B_i).$$
Note that $V^1_1$ is the star neighborhood of the $(n-k+1)$-skeleton  of $(V_1^0)$ in $\beta^2V^0_1$ which deforms onto it. The set $V^1_1$
is $f$-categorical as a subset of $f$-categorical set. 

For $i\ge 2$ we define 
$V_i^1=V_i^0\cup B_i$. The st $B_i$ splits naturally in two sets $B_i^0\cup B^i_1$ where $B_i^0$ is the union of stars that do not intersect $V_i^0$ and $B^1_i$ is the union of stars having nonempty intersection with $V_i^0$. If $$St(b_\sigma, \beta^2V^0_1)\cap V_i^0\ne\emptyset,$$ then $\sigma\subset V_1^0\cap V_i^0$. Then the intersection $$St(b_\sigma, \beta^2V^0_1)\cap V_i^0=St(b_\sigma,\beta^2(V_1^0\cap V_i^0))$$ is contractible. Moreover,
the star $St(b_\sigma, \beta^2V^0_1)$
can be deformed to that intersection. Therefore, $V_i^0\cup B_i^1$ admits a deformation retraction onto $V^0_i$. Hence $V_i^0\cup B_i^1$ is $f$-categorical.
The set $V_i^1$ is $f$-categorical as a disjoint union of the $f$-categorical set and the disjoint union of finitely many contractible sets.
\end{proof}
\begin{remark}
We can apply  this theorem to the case when each set $V_i^0\subset X$ is $f_i$-categorical for its own map $f_i:X\to Y_i$. The proof works without changes.
\end{remark}

\

\subsection{Pasting sections}
We recall that a map $p:E\to B$  satisfies the {\em Homotopy Lifting Property for a pair} $(X,A)$ if for any
homotopy $H:X\times I\to B$ with a lift $H':A\times I\to E$ of the restriction $H|_{A\times I}$
and a lift $H_0$ of $H|_{X\times0}$ which agrees with $H'$, there is a lift $\bar H:X\times I\to E$
of $H$ which agrees with $H_0$ and $H'$. We recall that a pair of spaces $(X,A)$ is called an NDR pair if $A$ is a deformation retract of a neighborhood in $X$. In particular, every CW complex pair is an NDR pair.  It is well-known~\cite{tD}, Corollary 5.5.3 that any Hurewicz fibration $p:E\to B$  satisfies the Homotopy Lifting Property for NDR pairs $(X,A)$.

\begin{lemma}\label{DS}\cite{DS}
Let $p:E\to B$ be a Hurewicz fibration over a CW complex $B=X\cup Y$ 
presented as the union of subcomplexes whose intersection  $C=X\cap Y$ has $\dim C \le n$. Suppose that there are sections of $p$ over $X$ and $Y$. Then $p$ admits a section $s:B\to E$ in the following cases:

(1) when the fiber $F$ of $p$ is $n$-connected;

(2) when $F$ is $(n-1)$-connected with $H^n(C;\mathcal F)=0$
for any local coefficients, $n>0$.
\end{lemma}

\section{Connectivity of  $\mathcal{B}_n(X)$}

It was shown in~\cite{KK}, Theorem 1.2,  that for simply connected CW-complex $X$  the space $\mathcal B_n(X)$ is
$(2n-1)$-connected.  The goal of this section to prove 

\begin{thm}\label{connectivity}
For any CW complex $X$ the space $\mathcal{B}_n(X)$ is $(n-2)$-connected.
\end{thm}

We will be using the following
\begin{thm}[\cite{KK}, Theorem 3.6]\label{3.6} For any connected CW complex $X$ the space
$\mathcal{B}_n(X)$ is simply connected for all $n\ge 2$.
\end{thm}
Also, we will be using the following theorem the proof of which is postponed to the end of this section.

\begin{thm}\label{connectivity1}
For any connected CW complex $X$ the space $\mathcal{B}_n(X)$ is $(n-1)$-connected.
\end{thm}
Besides Theorem~\ref{3.6} and Theorem~\ref{connectivity1}
the following well-known fact is used in the proof: The union of $n$-connected spaces $X=A\cup B$ is $n$-connected if the intersection $A\cap B$ is $(n-1)$-connected, where $(-1)$-connectivity means that the space is not empty.
This holds true for reasonably nice sets like CW-subcomplexes. It's known that it suffices for them to be good in a sense of~\cite{Ha}, or equally to form NDR pairs.

Finally we will be using the following version pof Vietoris-Begle theorem.

\begin{thm}[Combinatorial Vietoris-Begle Theorem]~\label{VB} Let $p:X\to K$ be a proper map onto a finite simplicial complex such that $p^{-1}(\sigma)$ is an ANR
 and $H_i(p^{-1}(\sigma))=0$ for $i\le n$ for every simplex $\sigma\subset K$.
Then $p_*:H_i(X)\to H_i(K)$ is an isomorphism for $i\le n$.
\end{thm}
\begin{proof}
This is well-known fact which can be proven by induction on dimension by means of the Mayer-Vietoris exact sequence. It also can be derived from homology version of Vietoris-Begle theorem by modifying the map $p$ to a homotopy equivalent map $p':X'\to K$ with acyclic up to dimension $n$ point preimages. We refer to~\cite{Dy} for homology version of the Vietoris-Begle
theorem.
\end{proof}

Let $X$ be a CW complex and $q:X\to C$ be the quotient map collapsing path components to points.
We denote by $$q_n=\mathcal B_n(q):\mathcal B_n(X)\to\mathcal B_n(C)=\Delta(C)^{(n-1)}$$ the induced map on the $n$-measures
where $\Delta(C)$ is the simplex spanned by $C$. 
Thus, $q_1=q$. 
By a slight abusing of notations we use the same  symbol for the map $q_n:\mathcal B_n(X)\to\Delta(C)$.

By Theorem~\ref{3.6}, the preimages $q_n^{-1}(v)$  are simply connected  for $n\ge 2$ for all vertices $v\in C$.

\begin{prop}\label{preimages}
For any simplex $\sigma\subset\Delta(C)$ the map $q_n$ induces isomorphisms of homotopy groups
$$(q_n)_*:\pi_i(q_n^{-1}(\sigma))\to\pi_i(\sigma^{(n-1)})$$   for $i\le n-1$ for all $\sigma$.
\end{prop}
\begin{proof}
We prove it by induction on $n$. It is a true statement for $n=1$, but still we need to treat the case $n=2$ separately.
By Theorem~\ref{3.6} the space $q_2^{-1}(v)$ is simply connected.
It is an easy exercise to show that $q_2^{-1}([v_0,v_1])$ is simply connected. Then $q_2^{-1}(\partial\Delta^2)$ for a 2-simplex $\Delta^2$ has the fundamental group equal $\mathbb Z$ which maps isomorphically onto $\pi_1(\partial\Delta^2)$ by $q_2$. Let a simplex $\sigma$ be of $\dim\sigma>2$ and let $v\in\sigma$ be a vertex. Let $\{\Delta_i\}$ be the family of all
2-dimensional faces containing $v$. Then $\sigma^{(1)}=\cup_i\Delta_i^{(1)}$. Thus,
we have the cover of $q_2^{-1}(\sigma)=q_2^{-1}(\sigma^{(1)}$
by the sets $q_2^{-1}(\Delta_i)$. Note that double and triple intersections $q_2^{-1}(\Delta_i)\cap  q_2^{-1}(\Delta_j)$ and $q_2^{-1}(\Delta_i)\cap  q_2^{-1}(\Delta_j)\cap  q_2^{-1}(\Delta_k)$ are simply connected. Then by van Kampen theorem~\cite{Ha} $q_2$ induces an isomorphism of the fundamental groups
$(q_2)_*:\pi_1(q_n^{-1}(\sigma))\to\pi_1(\sigma^{(1)})$.

Assume that the statement of proposition holds true for some $n\ge 2$. First we prove it for $n+1$ for simplices of dimension $N\le n$ by induction on $N$. 
In that case our statement turns into the following: The preimage $q_{n+1}^{-1}(\sigma)$ is $n$-connected.
When $\sigma=v$ is a vertex $v\in C$, 
the preimage $q^{-1}_{n+1}(\sigma)=\mathcal B_{n+1}(q^{-1}(v))$ is $n$-connected by Theorem~\ref{connectivity1}.
Suppose that for $ \sigma$ of dimension $N<n$ the preimage $q^{-1}_{n+1}(\sigma)$ is $n$-connected.
Let $\sigma'=\sigma\ast v$ be a simplex of dimension $N+1$. Note that $$q_{n+1}^{-1}(\sigma\ast v)=\bigcup_{k=0}^{n+1}q_k^{-1}(\sigma)\ast q_{n-k+1}^{-1}(v)$$
where we use the convention $q_0=\emptyset$, and $\emptyset\ast q_{n+1}^{-1}(v)= q_{n+1}^{-1}(v)$. 
 By the external induction assumption  we obtain that $q_k^{-1}(\sigma)=q_k^{-1}(\sigma^{(k-1)})$ is $(k-2)$-connected. Then
each space $M_k=q_k^{-1}(\sigma)\ast q_{n-k+1}^{-1}(v)$ for $1\le k\le n$ is $n$-connected as the join product of $(k-2)$-connected and $(n-k)$-connected spaces. The spaces $M_0=q_{n+1}^{-1}(v)$ and $M_{n+1}=q_{n+1}^{-1}(\sigma)$ are $n$-connected by Theorem~\ref{connectivity1} and the internal induction assumption.
Note that $$(\cup_{i=0}^kM_i)\cap M_{k+1}=q_{k-1}^{-1}(\sigma)\ast q^{-1}_{n-k+1}(v)$$ is 
$(n-1)$-connected for each $k$. Thus, $q^{-1}_{n+1}(\sigma')=\cup_{k=0}^{n+1}M_k$ is $n$-connected.

Next we show by induction on $\dim\sigma'$ that $q_{n+1}^{-1}(\sigma')$ is simply connected. 
By what we have proven above this is true when $\dim\sigma'\le n$. We consider $\sigma'=\sigma\ast v$ where $q_{n+1}^{-1}(\sigma)$ is 1-connected.
Since $(\sigma')^{(n)}=\sigma^{(n)}\cup(\sigma^{(n-1)}\ast v)$, we obtain $$q_{n+1}^{-1}(\sigma')=q_{n+1}^{-1}(\sigma^{(n)}\cup(\bigcup_{k=0}^{n}q^{-1}_{k}(\sigma^{(n-1)})\ast q_{n-k}^{-1}(v)).$$
 We show that $M=\bigcup_{k=0}^{n}M_k$ where $M_k=q^{-1}_{k}(\sigma^{(n-1)})\ast q_{n-k}^{-1}(v)$
is simply connected. First we note that each $M_k$ for $1<k$ being the join product, where one of the  factors is connected, is simply connected. $M_0=q_{n+1}^{-1}(v)$ is 1-connected by Theorem~\ref{3.6}. Arguing as above we can show that $M$ is 1-connected. Note that $q_{n+1}^{-1}(\sigma^{(n)})\cap M=q^{-1}_{n+1}(\sigma^{(n-1)}$ is connected for $n\ge 2$.
Thus, $q_{n+1}^{-1}(\sigma')$ is simply connected.

In the case when $\dim\sigma\ge n+1$ the $n$-skeleton $\sigma^{(n)}=\cup_i\sigma^n_i$ is partitioned into $n$-dimensional faces. Then $q_{n+1}^{-1}(\sigma)$ is partitioned into
the sets $P_i=q_{n+1}^{-1}(\sigma^n_i)$. We proved above that the sets $P_i$ and their all nonempty intersections are $n$-connected. Then by the Combinatorial Vietoris-Begle theorem
(Theorem~\ref{VB})
$q_n$ induces isomorphisms of homology groups $$(q_n)_*:H_i(q_n^{-1}(\sigma))\to H_i(\sigma^{(n-1)})$$   for $i\le n$. Since the spaces there are simply connected,
we complete the external induction by  the Hurewicz theorem.

We omit the proof that all above spaces are good in sense of~\cite{Ha}.
\end{proof}

\

Since $\mathcal B_n(X)=q_n^{-1}(\Delta(C)^{(n-1)})$, we obtain Theorem~\ref{connectivity} as a corollary of Proposition~\ref{preimages}.

\

\subsection{Bigraded ring $\mathcal R(X)$.}
For a topological space $X$ we define $$R_{i,m}(X)=H_i(SP^m(X),SP^{m-1}(X))$$ and let $$\mathcal R(X)=\bigoplus_{i,m}R_{i,m}(X).$$
The natural multiplication
$$
M: SP^m(X)\times SP^n(X) \to SP^{m+n}(X)$$
is defined by the formula
$$
M([x_1,\dots,x_m],[y_1,\dots y_n])=[x_1,\dots, x_m,y_1,\dots y_n].$$
For homology we obtain $M_*(R_{i,m}\otimes R_{s,t})\subset R_{i+s,m+t}$. That gives
a bigraded ring structure on $\mathcal R(X)$.

For a finite simplicial complex $X$ Milgram defined in~\cite{Mi} a bigraded differential algebra $C(X)$ which is a chain complex for
$SP^\infty(X)$ satisfying the following properties.
the following.
\begin{thm} [\cite{Mi}, Theorem 5.1]~\label{Mi}
There are bigraded ring isomorphisms
$$
\mathcal R(X)\cong H_*(C(X))
$$
and, moreover,
$$
\mathcal R(X\vee Y)\cong H_*(C(X)\otimes C(Y)).$$
\end{thm}

\begin{defn}
We call a bigraded ring $A=\oplus A_{i,m}$ {\em $r$-connected} if $A_{i,m}=0$ for $i< m+r$.
\end{defn}
\begin{prop}\label{tensor}
The tensor product and the Tor product of 1-connected and 0-connected bigraded rings is 1-connected.
\end{prop}
\begin{proof}
Let $A$ be 1-connected and $B$ be 0-connected.
By definition $$(A\otimes B)_{i,j}=\bigoplus_{k+r=i,s+t=j}A_{k,s}\otimes B_{r,t}$$ and 
$$Tor(A,B)_{i,j}=\bigoplus_{k+r=i,s+t=j}Tor(A_{k,s},B_{r,t}).$$
Whenever $k> s$ and $r\ge  t$, we have $i> j$. Hence for $i\le j$ we have $(A\otimes B)_{i,j}=0$ and $Tor(A,B)_{i,j}=0$.
\end{proof}

\begin{prop}\label{connected wedge}
If bigraded rings $\mathcal R(X)$ and $\mathcal R(Y)$ are 1-connected, then so is $\mathcal R(X\vee Y)$.
\end{prop}
\begin{proof}
By Theorem~\ref{Mi}, the Kunneth formula, and Proposition~\ref{tensor} $$R_{i,m}(X\vee Y)=(\mathcal R(X)\otimes\mathcal R(Y))_{i,m}\oplus Tor(\mathcal R(X),\mathcal R(Y))_{i-1,m}=0$$
for $i<m$.
\end{proof}

\begin{prop}\label{h.e.}
The probability space $\mathcal B_n(X)$ is homotopy equivalent to the symmetric join $Symm^{\ast n}(X)$.
\end{prop}
\begin{proof}
There is a proper cell-like map $$Symm^{\ast n}(X)\to\mathcal B_n(X).$$
\end{proof}

\begin{thm}[\cite{KK}, Theorem 1.3]\label{KK}
The reduced suspension $\Sigma Symm^{\ast n}(X)$ is homeomorphic to $\overline{SP}^n(\Sigma X):=SP^n(\Sigma X)/SP^{n-1}(\Sigma X)$.
\end{thm}
\begin{prop}\label{Hurewicz}
For any connected CW complex $X$ for $n\ge 2$ the following are equivalent:

\rm{(1)} $\mathcal{B}_n(X)$ is $(n-1)$-connected;

\rm{(2)}  $Symm^{\ast n}(X)$ is $(n-1)$-connected;

\rm{(3)}  $\Sigma Symm^{\ast n}(X)$ is $n$-connected;

\rm{(4)} $H_i(SP^n(\Sigma X), SP^{n-1}(\Sigma X))=0$ for $i\le n$;

\rm{(5)} $\mathcal R(\Sigma X)$ is 1-connected.

\end{prop}
\begin{proof}
From Proposition~\ref{h.e.} and Theorem~\ref{KK} we obtain (1) $\Leftrightarrow$ (2) $\Leftrightarrow$ (3) $\Rightarrow$ (4) where the implication (2) $\Leftarrow$ (3) follows from the fact that
the space $Symm^{\ast n}(X)$ is simply connected connected for connected $X$~\cite{KK}.
The implication (3) $\Leftarrow (4)$ follows from the Hurewicz theorem and Theorem~\ref{KK}.
The equivalence (4) $\Leftrightarrow$ (5) is the definition.
\end{proof}

\begin{cor}\label{ex}
The graded ring $\mathcal R(\Sigma X)$ is 1-connected whenever the space $X$ is 1-connected.
\end{cor}
\begin{proof}
Indeed, for simply connected $X$  the space $\mathcal B_n(X)$ is $(2n-1)$-connected
by Theorem 1.2 of ~\cite{KK}. Then by Proposition~\ref{Hurewicz} $\mathcal R(\Sigma X)$ is 1-connected.
\end{proof}

In view of Proposition~\ref{Hurewicz} to prove Theorem~\ref{connectivity1} it suffices to show that the bigraded ring $\mathcal R(\Sigma X)$ is 1- connected.
In~\cite{D} Dold proved that for finite complexes  $\mathcal R(X)=\mathcal R(Y)$ whenever $H_*(X)=H_*(Y)$.
By the Dold theorem  one can replace the connected space $X$ by the finite wedge of Moore spaces
$$Y= M(H_2(\Sigma X,2)\vee\dots \vee M(H_k(\Sigma X), k)$$ where
$k=\dim X+1$. Note that $ M(H_3(\Sigma X,3)\vee\dots \vee M(H_k(\Sigma X), k)=\Sigma Z$ with simply connected $Z$.

Thus, in view of Proposition~\ref{connected wedge} and Corollary~\ref{ex} it suffices to prove
that $\mathcal R(M(H_1(X),2))$ is 1-connected. Since our complex $X$ is finite, Proposition~\ref{connected wedge} reduces this  to the problem of 1-conectedness of the bigraded rings
$\mathcal R(S^2)$ and $\mathcal R(M(\mathbb Z_h,2))$, where $h=p^k$ for prime $p$.
Since $SP^n(S^2)=\mathbb CP^n$ and the inclusions $\xi^n_{n+1}:\mathbb CP^n\to \mathbb CP^{n+1}$ are the natural inclusions, $\mathcal R(S^2)$ is 1-connected.
\begin{prop}
The bigraded ring $\mathcal R(M(\mathbb Z_h,2))$, where $h=p^k$ for prime $p$, is 1-connected.
\end{prop}
\begin{proof}
We may assume that $M(\mathbb Z_h,2)=\Sigma(S^1\cup_{p^k}B^2)$.

For each Moore space $M(G,n)$ with finitely generated group $G$
Milgram constructed in~\cite{Mi} a bigraded differential algebras whose homology bigraded
isomorphic to the rings $\mathcal R(M(G,n))$.
In particular,
$$\mathcal R(\Sigma(S^1\cup_{p^k}B^2)\cong H_*(B(E_{h}(1,1)))$$ and $\mathcal R(S^1)\cong H_*(E(1,1))$,
where the algebra $E(1,1)$ is generated by one element $e\in E(1,1)_{1,1}$ with $e\cdot e=0$.
Clearly, it is 0-connected. 
If one forget about the differential part, the bigraded algebra $E_{h}(1,1)$ is isomorphic to the tensor product $E(1,1)\otimes P(2,1)$
where $P(2,1)$ is the divided polynomial algebra  generated by a single generator
in $P(2,1)_{2i,i}$ for each $i$ and, hence, is 1-connected. Since $E(1,1)$ is 0-connected, by Proposition~\ref{tensor}
we obtain that $E_{h}(1,1)$ is 1-connected
bigraded algebras.

If $A$ is an augmented bigraded ring, i.e., there is a bigraded ring homomorphism
$\epsilon:A\to\Lambda$ where $\Lambda_{0,0}=\mathbb Z$ and $\Lambda_{i,j}=0$
for all other indexes. Let $\bar A$ be the augmentation ideal. Then by definition~\cite{Mi}
$$
B(A)=\Lambda\oplus\bar A\oplus(\bar A\otimes\bar A)\oplus(\bar A\otimes\bar A\otimes\bar A)\oplus\dots$$
where  the bigrading is done by the following rule: $a_1\otimes\cdots\otimes a_n\in (\bar A\otimes\cdots\otimes\bar A)_{i,j}$ 
belongs to $BA_{i+n,j}$. It is an easy observation that if $A$ is 1-connected then so is $B(A)$.
Since homology of 1-connected bigraded differential algebra is a bigraded 1-connected ring, the result follows.
\end{proof}
We obtain the following
\begin{prop}\label{final}
For a finite connected  CW-complex $X$ the graded ring $\mathcal R(\Sigma X)$ is 1-connected and, hence, 
$\mathcal B_n(X)$ is $(n-1)$-connected.
\end{prop}
{\em Proof of Theorem~\ref{connectivity1}.} 
Since $X$ is a direct limit of finite connected CW-complexes,
in view of the equality $$\pi_i(\mathcal B_n(X))=\lim_{\rightarrow}\pi_i(\mathcal B_n(X_k)),$$
we obtain that $\mathcal B_n(X)$ is $(n-1)$-connected. \qed

\section{$d\TC$ of spaces}

\begin{prop}\label{CW}
Suppose that $d\cat X\le n$ for an aspherical CW-complex $X$, then $d\cat X^{(n+1)}\le n$.
\end{prop}
\begin{proof}
In view of Theorem~\ref{chardcat} in the pull-back  diagram
\[
\xymatrix{P_0(X^{(n+1)})_{n+1}\ar[dr]^{\mathcal B_{n+1}(p_0)}\ar[r]^q 
& j^*P_0(X)_{n+1}\ar[d]^{\xi}\ar[r]^{j'} & P_0(X)_{n+1}\ar[d]^{\mathcal B_{n+1}(p_0^X)}\\
& X^{(n+1)}\ar[r]^j & X}
\]
the map $\xi$ has a section $s:X^{(n+1)}\to j^*P_0(X)_{n+1}$. 
The map $q$ between the fibers of $\mathcal B_{n+1}(p_0)$ and $\xi$ is the inclusion $q|:{\mathcal B}_{n+1}(\Omega X^{(n+1)})\to \mathcal B_{n+1}(\Omega X)$.
We note that the loop space on CW-complex is homotopy equivalent to a CW-complex. Since $X$ is aspherical, the loop space $\Omega X$ is homotopy equivalent to a discrte space $C$.
Then the restriction $q|$ is homotopy equivalent to the map $q_{n+1}:{\mathcal B}_{n+1}(\Omega X^{(n+1)})\to \mathcal B_{n+1}(C)$ from Proposition~\ref{preimages}.
By Proposition~\ref{preimages}, $q_{n+1}$ induces an isomorphism of  homotopy groups 
in dimension $\le n$. Since $q_{n+1}$ admits a section, it induces an epimorphism of homotopy groups in all dimension. Then $q|$ is an $n$-equivalence.
By the Five Lemma the map $q$ is an $n$-equivalence, i.e., it induces isomorphism of homotopy groups in dimension $\le n$ and an epimorphism in dimension $n+1$.
Then the map $s$ has a homotopy lift.
Since the map $\mathcal B_{n+1}(p_0):P_0(X^{(n+1)})\to X^{(n+1)}$ is a fibration~\cite{DJ} that admits a homotopy section, it admits a real section. By Theorem~\ref{chardcat} $d\cat X^{(n+1)}\le n$.
\end{proof}

\subsection{$d\TC$ of wedge sum}
Let $j:X\times Y\to (X\vee Y)\times(X\vee Y)$ be the inclusion.
\begin{prop}\label{h}
There is a fiber-wise embedding $h: P_0(X\times Y)\to P(X\vee Y)$  making the diagram 
$$
\begin{CD}
P_0(X\times Y) @>h>> P(X\vee Y)\\
@Vp_0VV @V\bar pVV\\
X\times Y @>j>> (X\vee Y)^2\\
\end{CD}
$$
commutative and a fiber-wise retraction $r:\bar p^{-1}(j(X\times Y))\to P_0(X\times Y)$.
\begin{proof}
Let $(v,v)\in X\times Y$ be the base point and $\phi=(\phi_X,\phi_Y):I\to X\times Y$ be a path ending at it. We define a map $$h: P_0(X\times Y)\to P(X\vee Y)$$ by the formula $h(\phi)=\phi_X\cdot\bar\phi_Y$. 
Then the diagram commutes:
$$jp_0(\phi)=(\phi_X(0),\phi_Y(0))=(\phi_X\cdot\bar\phi_Y(0),\phi_X\cdot\bar\phi_Y(1))=\bar p(\phi_X\cdot\bar\phi_Y)= \bar ph(\phi).$$
We define a map $r:P(X\vee Y)\to P_0(X\times Y)$ as $r(\psi)=(r_X\psi,\overline{r_Y\psi})$ where $r_X:X\vee Y\to X$ and $r_Y: X\vee Y\to Y$ are the collapsing maps.
It is easy to check that $r$ is a retraction:
$$rh(\phi_X,\phi_Y)=h(\phi_x\cdot\bar\phi_Y)=(r_X\phi_X\cdot\bar\phi_Y,\overline{r_Y\phi_X\cdot\bar\phi_Y})=(\phi_X,\phi_Y).$$
\end{proof}
\end{prop}
\begin{thm}\label{wedge}
The equality
$$d\TC(X\vee Y)=\max\{d\TC (X),d\TC (Y),d\cat(X\times Y)\}$$
holds for CW complexes $X$ and $Y$
whenever $$\max\{\dim X,\dim Y\}<\max\{d\TC (X),d\TC (Y),d\cat(X\times Y)\}.$$
\end{thm}
\begin{proof}
 Let $v\in X\vee Y$ be the wedge point and let $r_X:X\vee Y\to X$ and $r_Y:X\vee Y\to Y$ be the retractions collapsing $Y$ to $v$ and $X$ to $v$, respectively. The inequality $$d\TC(X\vee Y)\ge\max\{d\TC (X),d\TC (Y)\}$$ holds in view of retractions $r_X$ and $r_Y$.
Let $d\TC(X\vee Y)=n$. We show that $d\cat(X\times Y)\le n$.
Let $$s:(X\vee Y)\times(X\vee Y)\to\mathcal B_{n+1}(P(X\vee Y))$$ be a $(n+1)$-distributed
navigation algorithm, $s(z,z')=\sum\lambda_\phi\phi$
where $\phi:I\to X\vee Y$ is a path from $z$ to $z'$. We define an $(n+1)$-contraction $$H:X\times Y\to \mathcal B_{n+1}(P_0(X\times Y))$$ of $X\times Y$ to the point $(v,v)$ by the formula
$$
H(x,y)=\sum_{\phi\in\supp s(x,y)}\lambda_\phi(r_X\phi,r_Y\bar\phi)
$$
where $\bar\phi$ is the reverse path from $y$ to $x$. Clearly, $r_X\phi:I\to X$ is a path from $x$ to $v$ and $r_Y\bar\phi:I\to Y$ is a path from $y$ to $v$. Hence, $(r_X\phi,r_Y\bar\phi):I\to X\times Y$ is a path from $(x,y)$ to $(v,v)$.

Next, we prove the inequality
$$d\TC(X\vee Y)\le\max\{d\TC (X),d\TC (Y),d\cat(X\times Y)\}.$$
Let $n=\max\{d\TC (X),d\TC (Y),d\cat(X\times Y)\}$. We construct a section $s$ of the fibration $\mathcal B_{n+1}(\bar p)$ with $$\bar p:P(X\vee Y)\to(X\vee Y)\times(X\vee Y)=(X\times X)\cup (Y\times Y)\cup (X\times Y)\cup (Y\times X).$$
Let $c_v:I\to X\vee Y$ denote the constant loop $c_v(t)=v$. We define $s(v,v)$ to be the Dirac measure $\delta_{c_v}$ identified with $c_v$ in our notations. 
Note that the fibration $$\mathcal B_{n+1}(\bar p^X) :P(X)_{n+1}\to X \times X$$ defined for
$
\bar p^{X}:P(X)\to X\times X
$
is naturally embedded in the fibration $\mathcal B_{n+1}(\bar p)$. Since $d\TC(X)\le n$, by Proposition~\ref{chardTC}, there is a section $$s_X:X\times X\to P(X)_{n+1}$$ of  $\mathcal B_{n+1}(\bar p^X)$,
which is also a section of $\mathcal B_{n+1}(\bar p)$. We may assume that $s_X(v,v)=s(v,v)$. Similarly, there is a section $$s_Y:Y\times Y\to P(X\vee Y)_{n+1}$$ of $\mathcal B_{n+1}(\bar p)$
with $s_Y(v,v)=s(v,v)$.

The map $h$ from Proposition~\ref{h} defines a fiberwise map $$h_{n+1}:P_0(X\times Y)_{n+1}\to P(X\vee Y)_{n+1}.$$
Since $d\cat (X\times Y)\le n$, by Proposition~\ref{chardcat}, there is a section $\sigma_{X\times Y}$ of $\mathcal B_{n+1}(p_0^{X\times Y})$.
Then $$s'_{X\times Y}=h_{n+1} \hspace{1mm} \sigma_{X\times Y} : X \times Y \to P(X \vee Y)_{n+1}$$ is a section of $\mathcal B_{n+1}(\bar p)$ over $X\times Y$. On the set $X\vee Y=X\times v\cup v\times Y$, this section $s'$ could disagree with $s_Y\cup s_X$. We will correct it as follows.
First we note that the loop space on a CW complex is homotopy equivalent to a CW complex.
Then by Theorem~\ref{connectivity} we obtain that the  space $F=\mathcal B_{n+1}(\Omega(X\vee Y)),$  which is the fiber of the fibration $\mathcal B_{n+1}(\bar p)$,
is $(n-1)$-connected. Therefore, since $\dim(X\vee Y)<n$, there is a fiberwise deformation of $s'_{X\times Y}$ to a section $s_{X\times Y}$ that agrees with $s_Y\cup s_X$ (see Lemma~\ref{DS}).
Similarly, we construct a section $s_{Y\times X}$ that agrees with $s_Y\cup s_X$. Then $$s=s_Y\cup s_X\cup s_{X\times Y}\cup s_{Y\times X}$$ is a continuous section of $\mathcal B_{n+1}(\bar p)$. By Proposition~\ref{chardTC}, $$d\TC(X\vee Y)\le n.$$
\end{proof}

\begin{theorem}\label{modified}
The equality
$$d\TC(X\vee Y)=\max\{d\TC (X),d\TC (Y),d\cat(X\times Y)\}$$
holds for CW complexes $X$ and $Y$
whenever $$\max\{\dim X,\dim Y\}\le\max\{d\TC (X),d\TC (Y),d\cat(X\times Y)\}=n$$
and $H^n(X;\mathcal F)=0=H^n(Y;\mathcal G)$ for all local coefficients.
\end{theorem}
\begin{proof} If $n> \max\{\dim X,\dim Y\}$, the result follows from Theorem~\ref{wedge}. We may assume that $n=\max\{\dim X,\dim Y\}$.
To make the argument of Theorem~\ref{wedge} work here
we need to check the inequality $\TC(X\vee Y)\le n$. By Theorem~\ref{connectivity} the fiber $F$ is  $(n-1)$-connected. The fiberwise homotopy as above can be arranged on the $(n-1)$-skeleton
$X^{(n-1)}\vee Y^{(n-1)}$. The primary obstruction to construct it on the $n$-skeleton is zero, since $n$-cohomology groups of $X$ and $Y$ are trivial.
 Then the fibration $\mathcal B_{n+1}(\bar p)$ admits a section over $(X\vee Y)^2$.
\end{proof}

\subsection{$\TC$ of connected sum}
Let $X$ and $Y$ be closed $n$-manifolds with the orientation sheaves $\mathcal O_X$ and $\mathcal O_Y$. 
Suppose that $X$ and $Y$ have a common $n$-ball $B$. Let $\mathcal O'$ be the sheaf on $X\cup_BY$ obtained by
identification of $\mathcal O_X$ and $\mathcal O_Y$ along $\mathcal O_B$.
We denote by $\mathcal O$ the sheaf on the wedge $X\vee Y$ obtained 
by pulling back $\mathcal O'$ by means of a homotopy inverse  $X\vee Y\to X\cup_BY$ to the collapsing map $X\cup_BY\to (X\cup_BY)/B\cong X\vee Y$.
Then the restriction of the sheaf $\mathcal O\hat\otimes\mathcal O$ to any of the following manifolds $M\subset(X\vee Y)^2$ is the orientation sheaf, $M=X^2, Y^2, X\times Y,Y\times X$. The fundamental class $[X\# Y]$ defines the fundamental classes $[M]$ for all the above $M$.

\begin{prop}\label{cap}
Let $X$ and $Y$ be closed $n$-manifolds and $p:E\to(X\vee Y)^2$  be a fibration with $(2n-2)$-connected fiber such that $p$ admits sections over $X^2$ and $Y^2$
and does not admit a section over $X\times Y$. Suppose for the factor flipping map $\sigma:X\times Y\to V\times X$ satisfies $\sigma_*([X\times Y])=[X\times Y]$.
Then $$([X^2]+[Y^2]+[X\times Y]+[Y\times X])\cap\kappa\ne 0$$
for the primary obstruction $\kappa\in H^{2n}((X\vee Y)^2;\mathcal F)$ to a section of $p$ provided $2\kappa\ne 0$.
\end{prop}
\begin{proof}
We note that $$ (X\vee Y)^2=(X^2\vee Y^2)\cup ((X\times Y)\vee (Y\times X)).$$
Let $P=X^2\cup Y^2$ and $Q=(X\times Y)\vee (Y\times X)$.  Note that $P\cap Q=X\vee Y\vee X\vee Y$.
Consider the Mayer-Vietoris exact sequence for homology with coefficients in $\mathcal F\otimes \mathcal O\hat\otimes\mathcal O$
$$
H_0(P\cap Q)\stackrel{\phi}\to H_0(P)\oplus H_0(Q)\stackrel{\psi}\to H_0((X\vee Y)^2) \to 0.
$$
We recall that $\phi(x)=(-i_*^P(x),i_*^Q(x))$ and $\psi(a,b)=j_*^P(a)+j_*^Q(b)$ where $i^P:P\cap Q\to P$, $i^Q:P\cap Q\to Q$, $j^P:P\to(X\vee Y)^2$, and
$j^Q:Q\to(X\vee Y)^2$ are the inclusions. We use the notations $\kappa_P=(j^P)^*(\kappa)$ and $\kappa_Q=(j^Q)^*(\kappa)$. 
Then $\kappa_P$ and $\kappa_Q$ are the primary obstruction to sections over $P$ and $Q$ respectively.
 Note that $\psi=g_*$ where $g:P\coprod Q\to (X\vee Y)^2$ is the quotient map. Then  $$\psi(([X^2]+[Y^2])\cap \kappa_P, ([X\times Y]+[Y\times X])\cap\kappa_Q)=$$
$$
g_*(([X^2]+[Y^2]+[X\times Y]+[Y\times X])\cap g^*\kappa)=$$
 $$([X^2]+[Y^2]+[X\times Y]+[Y\times X])\cap\kappa=a.$$

By the hypotheses and  Proposition~\ref{liftCL} there is a section to $p$ over $P$. This implies that $\kappa_P=0$. Thus, $$\psi(0, ([X\times Y]+[Y\times X])\cap\kappa_Q)=a.$$
Since $Q=(X\times Y)\vee (Y\times X)$, we obtain $$([X\times Y]+[Y\times X])\cap\kappa_Q=\psi([X\times Y]\cap\kappa_{Q}\oplus [Y\times X]\cap\kappa_{Q})=$$
$$2[X\times Y]\cap\kappa_Q=[X\times]\cap 2\kappa_Q.$$
Since  $2\kappa_{Q}\ne 0$, by Poincare duality $[X\times]\cap 2\kappa_Q\ne 0$. 
Therefore, $$b=([X\times Y]+[Y\times X])\cap\kappa_Q\ne 0.$$ By definition of $\phi$ the pair $(0,b)$ is not in the image of $\phi$. Hence by exactness $a=\psi(0,b)\ne 0$.
\end{proof}

\begin{thm}\label{TClb}
Let $X$ and $Y$ be closed aspherical $n$-manifolds with $\cat(X\times Y)=2n$, $\TC(X),\TC(Y)<2n$, and 
$\sigma_*([X\times Y])=[X\times Y]$ for the flipping map. Then $$\TC(X\#Y)=2n.$$
\end{thm}
\begin{proof}
Let $q:X\#Y\to X\cup_BY$ be the inclusion. Then the orientation sheaf on $X\#Y$ is the restriction of the sheaf $\mathcal O$ defined above. Note that $q$ takes the fundamental class
$[X\#Y]$ to the sum $[X]+[Y]$ of the fundamental classes. Then $q\times q$ takes the fundamental class $[(X\#Y)^2]$ to $[X^2]+[Y^2]+[X\times Y]+[Y\times X]$.

By Theorem~\ref{wedgeCL} $\TC(X\vee Y)=2n$.
We show that $q\times q$ does not admit a lift with respect to $\bar p_{2n-1}:\Delta_{2n-1}(X\vee Y)\to (X\vee Y)^2$. 
We note that $\bar p_{2n-1}$ has $(2n-2)$-connected fiber. The condition $\TC(X),\TC(Y)<2n$ implies that $\bar p_{2n-1}$ admits sections over $X^2$ and $Y^2$.
By Proposition~\ref{h} there is a fiber-wise retraction over $X\times Y$ of  $\Delta_{2n-1}(X\vee Y)$ onto $G_{2n-1}(X\times Y)$.
Therefore, the condition $\cat(X\times Y)=2n$ implies that $\bar p_{2n-1}$ does not admit a section over $X\times Y$ Moreover, the primary obstruction is $\beta^{2n}$ where 
$\beta$ is the Berstein-Schwarz class for the group $\pi_1(X\times Y)$. Since $X\times Y$ is an aspherical manifold, the universality of $\beta$ implies that $\beta^{2n}$ has infinite order.
Thus, all conditions of Proposition~\ref{cap} are satisfied.

Let $\kappa\in H^{2n}((X\vee Y)^2;\mathcal F)$ be the primary obstruction to a section of $\bar p_{2n}$.
Then $(q\times q)^*(\kappa)$ is the primary obstruction to a lift. Since $q\times q$ induces an epimorphism of the fundamental groups, the homomorphism $$(q\times q)_*:H_0((X\#Y)^2;(q\times q)^*\mathcal A)\to
H_0((X\vee Y)^2;\mathcal A)$$ is an isomorphism for any coefficient system on $(X\vee Y)^2$. Then for coefficients in $\mathcal O\otimes\mathcal F$ we have $$(q\times q)_*([(X\#Y)^2]\cap(q\times q)^*(\kappa))=([X^2]+[Y^2]+[X\times Y]+[Y\times X])\cap\kappa.$$
By Proposition~\ref{cap} the Poincare dual to  $(q\times q)^*(\kappa)$ is nonzero. Thus, $(q\times q)^*(\kappa)\ne 0$ and, hence, there is no lift of $q\times q$.
By  Proposition~\ref{liftCL} $\TC(X\#Y)\ge \TC(q)\ge 2n$.
\end{proof}

\begin{thm}\label{dTClb}
If closed aspherical $n$-manifolds $X$ and $Y$ satisfy the conditions $d\cat(X\times Y)=2n$,  $d\TC(X),d\TC(Y)<2n$,  and 
$\sigma_*([X\times Y])=[X\times Y]$ for the flipping map. Then $$d\TC(X\#Y)=2n.$$
\end{thm}
\begin{proof}
Similarly to the proof of Theorem~\ref{TClb} we show that $q\times q$ does not admit a lift with respect to $\mathcal B_{2n}(\bar p):P(X\vee Y)_{2n} \to (X\vee Y)^2$. 
By Theorem~\ref{connectivity} the fiber of $\mathcal B_{2n}(\bar p)$ is $(n-2)$-connected. In view of Proposition~\ref{h} the condition $d\cat(X\times Y)=2n$ implies that there is no section over $X\times Y$
and the condition  $d\TC(X),d\TC(Y)<2n$ implies that sections exists over $X^2$ and $Y^2$.
Let $\kappa$ be the primary obstruction to a section of $\mathcal B_{2n}(\bar p)$.
Then $(q\times q)^*(\kappa)$ is the primary obstruction to a lift. Since $q\times q$ induces an epimorphism of the fundamental groups, $(q\times q)_*:H_0((X\#Y)^2;\mathcal A)\to
H_0((X\vee Y)^2;\mathcal A)$ is an isomorphism for any coefficient system. Note that for coefficients in $\mathcal O\otimes\mathcal F$
by Proposition~\ref{cap}$$(q\times q)_*([(X\#Y)^2]\cap(q\times q)^*(\kappa))=([X^2]+[Y^2]+[X\times Y]+[Y\times X])\cap\kappa $$ is nonzero.
Thus, the primary obstruction to a lift  $(q\times q)^*(\kappa)$ is nontrivial. Hence, there is no lift of $q\times q$.
By Proposition~\ref{liftD}. $d\TC(X\#Y)\ge d\TC(q)\ge 2n$.
\end{proof}

\section{$d\TC$ of discrete groups}
We recall that the Lusternik-Schnirelmann category and the Topological Complexity of a group $\Gamma$ are defined as $\cat(\Gamma)=\cat(B\Gamma)$ and $\TC(\Gamma)=\TC(B\Gamma)$. We define similarly $d\cat\Gamma=d\cat(B\Gamma)$ and $d\TC(\Gamma)=d\TC(B\Gamma)$. These definitions make sense, since all classifying spaces $B\Gamma$ of a group $\Gamma$
are homotopy equivalent and $d\cat(X)$ and $d\TC(X)$ are homotopy invariants.

We note that in both fibrations $p_0:P_0(X)\to X$ and $\bar p:P(X)\to X\times X$ the fibers are homotopy equivalent to the loop space $\Omega X$.
If $X$ is homotopy equivalent to $K(\Gamma,1)$ then the canonical quotient map $\Omega X\to\Gamma$ that collapses path components to points is a homotopy equivalence.
Thus, in the case of aspherical $X$ the fibration $p_0$ is fiber-wise homotopy equivalent to the universal covering map $u:\widetilde X\to X$ and $\bar p$ is equivalent to a coving map
$v:D(X)\to X\times X$ that corresponds to the diagonal subgroup of $\pi_1(X\times X)=\pi_1(X)\times\pi_1(X)$.
Then Propositions~\ref{chardcat},~\ref{chardTC} for groups turn into the following
\begin{prop}\label{chardcatg}
$d\cat(\Gamma) \le n$ if and only if the fibration $$\mathcal B_{n+1}(u):E\Gamma_{n+1}\to B\Gamma$$ admits a section.
\end{prop}
and
\begin{prop}\label{chardTCg}
$\dTC (\Gamma)\le n$ if and only if the fibration $$\mathcal B_{n+1}(v):D_{n+1}(\Gamma)\to B\Gamma\times B\Gamma$$ admits a section.
\end{prop}
Here $B\Gamma$ is an Eilenberg-MacLane complex $K(\Gamma,1)$ and $E\Gamma$ is its universal cover. Since the
universal covering $E\Gamma$ is the orbit space of the diagonal action
$$
\begin{CD}
\Gamma\times E\Gamma @>>> E\Gamma\\
@VVV @VuVV\\
\Gamma\times_\Gamma E\Gamma @>u>> B\Gamma\\
\end{CD}
$$
the above space $E\Gamma_{n+1}$ can be identified with the orbit space of the diagonal
action on  $\mathcal B_{n+1}(\Gamma)\times E\Gamma$ in
$$
\begin{CD}
\mathcal B_{n+1}(\Gamma)\times E\Gamma @>>> E\Gamma\\
@VVV @VuVV\\
\mathcal B_{n+1}(\Gamma)\times_\Gamma E\Gamma @>\mathcal B_{n+1}(u)>> B\Gamma.\\
\end{CD}
$$

We note that the fiber of both fibrations $\mathcal B_{n+1}(u)$ and $\mathcal B_{n+1}(v)$ is homeomorphic to $\mathcal B_{n+1}(\Gamma)$ which can be identified with the $n$-skeleton
of the infinite dimensional simplex $\Delta(\Gamma)$ spanned by $\Gamma$  which is supplied with the metric topology.

\subsection{Eilenberg-Ganea theorem for $d\cat$}

The classical Eilenberg-Ganea theorem states that for discrete groups $\cat\Gamma=\cd(\Gamma)$.
We recall that the cohomological dimension $\cd\Gamma$ of a group $\Gamma$ is the maximal number $n$ such that $H^n(\Gamma,M)\ne 0$ for some $\Z\Gamma$-module $M$.

\begin{thm}[Knudsen-Weinberger~\cite{KW}]\label{E-G for dcat}
$d\cat\Gamma=\cd\Gamma$ for torsion free groups $\Gamma$.
\end{thm}
\begin{proof}
The proof is identical with the proof of Theorem 7.4 from ~\cite{KW}. We bring it here, since it is short.

In view of the inequality $d\cat\Gamma \le\cat\Gamma$ and the equality $\cat\Gamma=\cd\Gamma$ we obtain $d\cat\Gamma\le\cd\Gamma$.
If $d\cat\Gamma=n$, then by Proposition~\ref{chardcatg} there is a section of the fibration $\mathcal B_{n+1}(u)$. 
Since the action of $\Gamma$ on $\Delta(\Gamma)$ is free, the fibration $q$ in the Borel construction  
$$
\begin{CD}
\Delta(\Gamma)^{(n)} @<<< \Delta(\Gamma)^{(n)}\times E\Gamma @>>> E\Gamma\\
@VVV @VVV @VVV\\
\Delta(\Gamma)^{(n)}/\Gamma @<q<< \Delta(\Gamma)^{(n)}\times_\Gamma E\Gamma @>\mathcal B_{n+1}(u)>> B\Gamma\\
\end{CD}
$$
 has contractible fiber $E\Gamma$, and hence, is  a homotopy equivalence. If $\cd\Gamma> n$, we have $H^{n+1}(\Gamma,M)\ne 0$ for some $\Z\Gamma$ module.
The existence of a section implies that $\mathcal B_{n+1}(u)^*$ is a nonzero homomorphism. Since $\Delta(\Gamma)^{(n)}\times_\Gamma E\Gamma $ is homotopy equivalent
to an  $n$-dimensional space, we obtain a contradiction.
\end{proof}

\subsection{$\dTC$ of the free product}

\begin{prop}\label{+one}
For all groups $cd(G\times H)\ge cd(G)+1$.
\end{prop}
\begin{proof}
We may assume that the groups have finite cohomological dimension. In particular, $H$ is torsion free. Hence it contains a copy of integers.
Since the cohomological dimension of a subgroup
does not exceed the cohomological dimension of a group~\cite{Br}, it follows
Then $$cd(G\times H)\ge cd(G\times\Z)\ge cd(G)+1.$$
\end{proof}

\begin{thm}\label{+two}~\cite{DS}
If a group $H$ is not free, then  $cd(G\times H)\ge cd(G)+2$ for all groups $G$.
\end{thm}

\begin{thm}\label{free}
The equality
$$d\TC(G\ast H)=\max\{d\TC (G),d\TC (H), \cd(G\times H)\}$$
holds for all torsion free groups  $G$ and $H$.
\end{thm}
\begin{proof}
First, we consider the case when there are classifying spaces $BG$ and $BH$ satisfying $\dim BG=\cd G$ and $\dim BH=\cd(H)$. By Proposition~\ref{+one} and Proposition~\ref{E-G for dcat} we obtain 
$$\max\{\dim BG,\dim BH\}<\cd(G\times H)=d\cat(G\times H)=d\cat(BG\times BH).$$ 
Thus, the condition of Theorem~\ref{wedge} is satisfied with $X=BG$ and $Y=BH$.
Then by Theorem~\ref{wedge}, in view of Proposition~\ref{E-G for dcat} and the definitions of $d\cat$ and $d\TC$ for groups, we obtain that 
$$d\TC(G\ast H)=\max\{d\TC(G), d\TC(H), \cd(G\times H)\}.$$

If $G$ is a hypothetical counter-example to the Eilenberg-Ganea conjecture, i.e. $\cd(G)=2$ and $\dim BG=3$, and  $\cd(H)\ge 2$,  then by Theorem~\ref{+two}  $\cd(G\times H)\ge \cd(G)+2$ and $\cd(G\times H)\ge \cd(H)+2$. Again we obtain that the condition of Theorem~\ref{wedge}
$$\max\{\dim BG,\dim BH\}<\cd(G\times H)=d\cat(G\times H)=d\cat(BG\times BH)$$ is satisfied.

\

Now we assume that $G$ is a hypothetical counter-example to the Eilenberg-Ganea conjecture and $\cd(H)=1$. Therefore, by Stallings theorem $H$ is a free group.
Then by the Eilenberg-Ganea theorem and Proposition~\ref{+one} we obtain
$$\cat(G\times H)=\cd(G\times H)\ge 3\ge \max\{\dim BG,\dim BH\}.$$ Then by Theorem~\ref{modified} $$d\TC(BG\vee BH)=\max\{d\TC(BG), d\TC(BH), d\cat(BG\times BH)\}$$
and the required equality follows.
\end{proof}

\subsection{Surface groups}

\begin{prop}\label{connected sum}
Let $M, N$ be two surfaces where $M$ is orientable. Then for the connected sum,
$$d\TC(M\#N)\ge d\TC(N).$$
\end{prop}
\begin{proof} We may assume that $d\TC(N)>3$. Then $d\TC(N)=4$.
Let $F$ be the fiber of fibration $\mathcal B_4(\bar p)$ for $N$. By Theorem~\ref{connectivity} the space $\mathcal B_4(F)$ is 2-connected.
Then the primary obstruction $\kappa\in H^4(N\times N;\mathcal F)$ to the section of $\mathcal B_4(\bar p):P(N)_4\to N\times N$ is nontrivial, $\kappa\ne 0$.
Since $M$ is orientable, we have the equality ${\mathcal O}_{M\# N}=q^*{\mathcal O}_N$ for  the orientation sheaves
where $q:M\#N\to N$ is a map collapsing $M$ to a point.
Then ${\mathcal O}_{(M\# N)\times(M\# N)}=(q\times q)^*{\mathcal O}_{N\times N}$.
By the Poincare Duality with local coefficients~\cite{Bre}  for $N\times N$ we obtain
$$
(q\times q)_*([(M\# N)\times(M\# N)]\cap (q\times q)^*\kappa)=[N\times N]\cap\kappa\ne 0
$$
where $[X]$ denotes a fundamental class of $X$.
Thus, we obtain $(q\times q)^*\kappa\ne 0$.
Hence $q\times q$ does not admit a lift with respect to $\mathcal B_4(\bar p)$. By Proposition~\ref{liftD}, $d\TC(M\#N)\ge d\TC(q)\ge 4$.
\end{proof}

\begin{thm}\label{nonor}
$d\TC(N_g)=4$ for $g>3$.
\end{thm}
\begin{proof}
If for the Klein bottle  $d\TC(K)=4$ by Proposition~\ref{connected sum} $$d\TC(M\#K)\ge d\TC(K)=4$$ where $M$ is orientable. This covers  all $N_g$ with $g>3$.

If $d\TC(K)=3$, then by Proposition~\ref{E-G for dcat} $$d\cat(K\times K)=\cd(\pi_1(K\times K))=4.$$ Theorem~\ref{dTClb} implies that $d\TC(K\#K)=4$. Then by Proposition~\ref{connected sum}
$$d\TC(T\#\dots\#T\#K\#K)\ge d\TC(K\#K)=4.$$ This covers all $N_g$ with even $g>3$. If $d\TC(K\#\mathbb RP^2)=4$, Proposition~\ref{connected sum} covers the case of odd $g>3$.

Assume that $d\TC((K\#\mathbb RP^2)=3$.
Proposition~\ref{covering} applied to the covering map $K\#K\to K\#\mathbb RP^2$ and Proposition~\ref{E-G for dcat} imply $$d\cat(K\times(K\#\mathbb R P^2))\ge d\cat(K\times (K\# K))=\cd(K\times (K\# K))=4.$$
Then by Theorem~\ref{dTClb}, $d\TC(K\#(K\#\mathbb RP^2))=4$. Then by Proposition~\ref{connected sum} $d\TC(N_g)\ge 4$ for odd $g>3$.
\end{proof}

\

\begin{question}
What are  the values of $d\TC(K)$ and $d\TC(N_3)$ ?
\end{question}

\subsection{Finite groups}
We note that for finite groups $\cat(G)=\TC(G)=\infty$. It is not the case for distributional invariants. It was shown in~\cite{DJ} and ~\cite{KW} that $d\cat(\mathbb Z_2)=d\TC(\mathbb Z_2)=1$. This computation is a special case of the following.

\begin{thm}[Knudsen-Weinberger~\cite{KW}]\label{KW}
For a finite group $G$, $d\cat(G)\le |G|-1$ and $d\TC(G)\le |G|-1$.
\end{thm}
\begin{proof}
Let $u:EG\to BG$ be the universal cover and let $n=|G|$.  Then the map $s: BG\to \mathcal B_n(EG)$  sending $x$ to the evenly distributed measure on $u^{-1}(x)$
defines a section of $\mathcal B_n(u)$. By Proposition~\ref{chardcatg} $d\cat G\le n-1$.

Similarly, the map $s:BG\times BG\to D_n(G)$ that sends each $x\in BG$  to the evenly distributed measure on $v^{-1}(x)$ defines  a section of $\mathcal B_n(v)$. By Proposition~\ref{chardTCg} $d\TC G\le n-1$.
\end{proof}

\begin{thm}\label{main 1}
Suppose that $p$ is prime, then $$d\cat(\mathbb Z_p)=d\TC(\mathbb Z_p)=p-1.$$
\end{thm}
\begin{proof}
Let $L_p^\infty=K(\mathbb Z_p,1)$ be the infinite lens space and $ L^{n}_p$ denote its $n$-dimensional skeleton for $n=p-1$.
We show that $d\cat(L^n_p)\ge n$. Assume the contrary, $d\cat(L^n_p)\le n-1$. Then by Lemma~\ref{SP} there is an open cover $\{U_i\}_{i=1}^n$ of $L^n_p$ such that each of the
diagonal inclusions $\delta_i:U_i\to SP^i(L^n_p)$ is null-homotopic. By Theorem~\ref{f-categorical} (see the remark after it) there is a partition $L^n_p=P_1\cup\dots\cup P_n$ into polyhedra such that $P_i\subset U_i$ and  each $P_i$ admits a deformation retract onto a complex $K_i$ of $\dim K_i\le 1$. We claim that the base point inclusion $\xi_i:L^n_p\to SP^i(L^n_p)$ restricted to $P_i$ is null-homotopic. Note that $K_i$ is homotopy equivalent to finite disjoint union of  points and wedges of circles $\vee^{k}S^1$. Let $h_i:\coprod  \vee_{j=1}^{k}S^1\to K_i$ be such homotopy equivalence. We show that $\xi_i h_i$ is null-homotopic.
It suffices to show that the restriction of $\xi_ih_i$ to each circle in every wedge is null-homotopic. By Proposition~\ref{times m} 
on the level of the fundamental groups we have the equality $[\delta_ih_i|_{S^1}]=i[\xi_ih_i|_{S^i}]$. Since $\delta_i$ is null-homotopic on $P_i$, we obtain the equality
$i[\xi_ih_i|_{S^i}]=0$ in the fundamental group $\pi_1(SP^i(L^n_p))=\mathbb Z_p$. Since $p$ is prime, this implies that $[\xi_ih_i|_{S^i}]=0$.
We obtain that $\cat(\xi_n)\le n-1$ which contradicts with Proposition~\ref{BerSchw}. Thus, $d\cat(L^n_p)\ge n$.

Assume that $d\cat(L^\infty_p)\le n-1$. Then
by Proposition~\ref{CW}, $d\cat(L^n_p)\le n-1$ and we obtain a contradiction.
Thus, 
$d\cat(\mathbb Z_p)=d\cat(L^\infty_p)\ge n=p-1$. Then in view of Theorem~\ref{KW} we obtain 
$$p-1\ge d\TC(\mathbb Z_p)\ge d\cat(\mathbb Z_p)\ge p-1.$$
\end{proof}

\section{Further generalization of $d\TC$}

For a metric space $X$ we call a path $f:[0,1]\to \mathcal B_n(X)$ {\em resolvable} if there is a continuous map $F:I\times\{1,\dots, n\}\to X$ and a probability measure $\mu$ on the set $\{1,\dots,n\}$
such that $f(t)=\mathcal B_n(F)(t\times\mu)$ for all $t\in I$.

We denote by $RP(\mathcal B_n(X))\subset P(\mathcal B_n(X))$ the subspace of resolvable paths.
Then the {\em $n$-intertwined navigation algorithm} on a space $X$ is a continuous map
$$
m:X\times X\to RP(\mathcal B_n(X))
$$
such that $m(x,y)(0)=\delta_x$ and $m(x,y)(1)=\delta_y$ for all $x,y\in X$.
We define {\em the intertwining topological complexity} $i\TC(X)$ of $X$ to be the minimal $n$ such that there is an $(n+1)$-intertwined navigation algorithm on $X$.

We define {\em the intertwining LS-category} $i\cat(X)$ of $X$ as the minimal number $n$ such that there is a continuous map
$h:X\to RP(\mathcal B_{n+1}(X))$ satisfying $h(x)(0)=\delta_x$ and $h(x)(1)=\delta_{x_0}$ for all $x\in X$~\cite{DJ}.

The following are straightforward.
\begin{prop}
$i\TC(X)\le d\TC(X)$.
\end{prop}
\begin{prop}
$i\TC(X)$ is homotopy invariant.
\end{prop}
\begin{cor}
The invariant $i\TC(\Gamma):=i\TC(K(\Gamma,1))$ is well-defined for discrete groups $\Gamma$.
\end{cor}
\begin{prop}
For the Higman group  we have $i\TC(H)=1$.
\end{prop}
\begin{proof}
Since $K=K(H,1)$ is acyclic, by Dold's theorem~\cite{D} $\tilde H_i(SP^2(K))=\tilde H_i(SP^2(pt))=0$ fro all $i$.
The fundamental group $\pi_1(SP^2(K))= H_1(K)=0$ . Then by Hurewicz theorem the space
$SP^2(K)$ is contractible. We denote by $[x,y]\in SP^2(K)$ the orbit of $(x,y)\in K\times K$ under permutation of coordinates.
Let $\delta:K\to SP^2(K)$, $\delta(x)=[x,x]$, denote the diagonal embedding.
Since $SP^2(K) $ is  a contractible CW complex~\cite{Mi}, it is an absolute extensor. Therefore, the map $K\coprod K\stackrel{\coprod\delta}\to SP^2(K)$,  $K\coprod K\subset_{\text{Cl}} K\ast K$, can be extended to a map  $f:K\ast K\to SP^2(K)$
of the join product. Let $q:K\times K\times[0,1]\to K\ast K$ be the quotient map
from the definition of the join product. The composition $f\circ q$ produces the map $\psi:K\times K\to P(SP^2(K))$ to the path space.
 We define an embedding $\phi:SP^2(K)\to\mathcal B_2(K)$ by the formula
$$\phi([x,y])=\frac{1}{2}x+\frac{1}{2}y.$$ The map $\phi$ induces the map of path spaces $$P(\phi):P(SP^2(K))\to P(\mathcal B_2(K)).$$ Note that every path in $SP^2(K)$ is the image
of a map $I\times\{1,2\}\to X$. It implies that $P(\phi)$ has its image in resolvable paths.
Then $$m=P(\phi)\circ\psi:K\times K\to RP(\mathcal B_{2}(K))$$ is a requested navigation algorithm.
\end{proof}
\begin{cor}
For the Higman group $i\cat H=1$, whereas $d\cat H=2$.
\end{cor}

\end{document}